\pdfoutput=1
\documentclass[10pt]{article}
\usepackage[letterpaper,hmargin=1.2in,vmargin=1.2in]{geometry} 
\usepackage{graphicx}
\usepackage{amsmath}
\usepackage{amsthm}
\usepackage{amssymb}
\usepackage{amsfonts}
\usepackage{latexsym}

\newtheorem{theorem}{Theorem}

\newtheorem{lemma}{Lemma}
\newtheorem{corollary}{Corollary}

\usepackage{multirow}

\usepackage[colorlinks=false,bookmarks=false,hyperfootnotes=false]{hyperref}
\def \endproof {\hfill $\Box$ \linebreak \smallskip}
\newcommand\bound{\alpha}
\newcommand\hi{\hspace{-1ex}}
\newcommand\hii{\hspace{-2ex}}
\newcommand\hiii{\hspace{-3ex}}

\makeatletter    
\renewcommand\paragraph{\@startsection{paragraph}{4}{\z@}%
                       {-12\p@ \@plus -4\p@ \@minus -4\p@}%
                       {-0.5em \@plus -0.22em \@minus -0.1em}%
                       {\normalfont\normalsize\itshape}}
\makeatother

\title{On the Number of Spanning Trees a Planar Graph Can Have}

\author { Kevin Buchin \thanks{Department of Mathematics and Computer Science, Technical University of Eindhoven,  {\tt k.a.buchin@tue.nl}, 
supported by the Netherlands Organisation for Scientific Research (NWO)
under project no. 639.022.707} \and Andr\'e Schulz
  \thanks{ Institut f\"ur Mathematsche Logik und Grundlagenforschung, Universit\"at M\"unster,  {\tt
      andre.schulz@uni-muenster.de}, supported by the German Research Foundation (DFG) under grant SCHU 2458/1-1.}}

\begin{document}
  \maketitle
  \begin{abstract}
We prove that any planar graph on $n$ vertices has less than $O(5{.}2852^n)$ spanning trees.
Under the restriction that the planar graph is 3-connected and contains no triangle and no quadrilateral the number of its spanning trees is less than $O(2{.}7156^n)$. As a consequence of the latter the grid size needed to realize a 3d polytope with integer coordinates can be bounded by $O(147.{7}^n)$. Our observations imply improved upper bounds for related quantities: the number of cycle-free graphs in a planar graph is bounded by  $O(6.4884^n)$, the number of plane spanning trees on a set of $n$ points in the plane is bounded by  $O(158.6^n)$, and the number of plane cycle-free graphs on a set of $n$ points in the plane is bounded by  $O(194{.}7^n)$.

\end{abstract}

\setcounter{page}{1}
\section{Introduction}


The number of spanning trees of a connected graph, also considered as the complexity of the graph, is an important graph invariant. Its importance largely stems from Kirchhoff's seminal matrix tree theorem: The number of spanning trees equals the absolute value of any cofactor of the Laplacian matrix of the graph. Furthermore, this number is the order of the Jacobian group of the graph, also known as critical group, or as sandpile model in theoretical physics~\cite{bhn-lif-97,b-aptg-97}. This group can be represented as a chip firing game on the graph; in this context the number of spanning trees counts the number of the stable and recurrent configurations~\cite{b-cfcg-99}. The number of spanning trees is also a measure for the global reliability of a network. Upper bounds on the number of combinatorial structures are often helpful to determine the (exponential) running time for exact algorithms of NP-complete problems.


Our motivation to study the number of spanning trees of planar graphs comes from an application of Kirchhoff's matrix tree theorem. Instead of computing the number of spanning trees with Kirchhoff's theorem one can use bounds on the number of spanning trees to obtain bounds for the cofactors of the Laplacian matrix. These cofactors appear in various settings. For example, Tutte's famous spring embedding is computed by solving a linear system that is based on the Laplacian matrix~\cite{t-crg-60,t-hdg-63}. As a consequence of Cramer's rule the cofactor of the Laplacian matrix is the denominator of all coordinates in the embedding. Therefore, by multiplying with the number of spanning trees, we can scale to an integer embedding. This idea finds applications in the grid embedding of 3d polytopes~\cite{rrs-epsg-07,rg-rsp-96}. Before we describe this application in more detail, we introduce some notation.

Let $\mathcal{G}_n$ be the set of all planar graphs with $n$ vertices. For a graph $G\in \mathcal{G}_n$ we denote the number of its (labeled) spanning trees with $t(G)$. For every $\mathcal{G}_n$ let  $T(n)$ be the maximal number of spanning trees a graph in this class can have, that is $T(n)=\max_{G\in\mathcal{G}_n}\{t(G)\}$. We study the growth rate of the function $T(n)$. Since it seems intractable to obtain an exact formula for $T(n)$, we aim at finding a value $\bound$ such that $T(n)\leq \bound^n$ for $n$ large enough. Notice that the graph that realizes the maximum $T(n)$ has to be a triangulation. Hence, it suffices to look at the subclass of all planar triangulations with $n$ vertices instead of considering all graphs in $\mathcal{G}_n$.

Furthermore, we are interested in the maximal number of spanning trees for planar graphs with special facial  structure. In particular, we want to bound
\begin{eqnarray*}
T_4(n)\!  & = &\! \max_{G\in\mathcal{G}_n}\{t(G) |  \text{$G$ is 3-connected and contains no triangle}\},\\
T_5(n)\! & = &\! \max_{G\in\mathcal{G}_n}\{t(G) |  \text{$G$ is 3-connected and contains no triangle or quadrilateral}\}.
\end{eqnarray*}

Notice that if a graph is planar and 3-connected its facial structure is uniquely determined~\cite{w-ig-31}. Let $\bound_4^n$ be an upper bound on $T_4(n)$ and $\bound_5^n$ be an upper bound on $T_5(n)$. We refer to the problem of bounding $\bound$ as the \emph{general problem}, and to the problems of bounding $\bound_4$ and $\bound_5$ as \emph{restricted problems}.

For embedding 3d polytopes
the necessary grid size (ignoring polynomial factors) can be expressed in terms of $\bound$, $\bound_4$  and $\bound_5$. In this scenario we are dealing with 3-connected planar graphs since $G$ is the graph of a 3d polytope~\cite{s-emw-22}. If the graph $G$ contains a triangle the grid size is in $O(\bound^{2n})$, if $G$ contains a quadrilateral the grid size is in $O({\bound_4}^{3n})$.  Due to Euler's formula every (3-connected) planar graph contains a pentagon -- in this case the grid size in $O({\bound_5}^{5n})$. As a consequence better bounds on $\bound$, $\bound_4$  and $\bound_5$ directly imply a better bound on the grid size needed to realize a polytope with integer coordinates.

Richter-Gebert used a bound on $T(n)$ to bound the size of the grid embedding of a 3d polytope~\cite{rg-rsp-96}. By applying Hadamard's inequality he showed that the cofactors of the Laplacian matrix of a planar graph are less than $6{.}5^n$. This bound can by easily improved to $6^n$ by noticing that the Laplacian matrix is positive semi-definite, which allows the application of the stronger version of Hadamard's inequality~\cite[page 477]{hj-ma-85}. Both bounds do not rely on the planarity of $G$, but on the fact that the sum of the vertex degrees  of $G$ is below $6n$.
Rib\'o and Rote improved Richter-Gebert's analysis and showed that  $5{.}0295^n\leq T(n) \leq 5{.}\bar{3}^n$~\cite{r-rcpps-06,r-nstpg-05}. The lower bound is realized on a wrapped up triangular grid  and was obtained by the transfer-matrix method.
For the upper bound they count the number of the spanning trees on the dual graph. This number coincides with the number of spanning trees in the original planar graph. Since  the number of spanning trees is maximized by a triangulation, the dual graph is $3$-regular. Applying a result of McKay~\cite{m-strg-83}, which bounds the number of spanning trees on $k$-regular graphs, yields the bound of $5{.}\bar{3}^n$. Interestingly, this bound is not directly related to the planarity of the graph. Therefore, Rib\'o and Rote tried to improve the bound using the \emph{outgoing edge approach}. The approach involves choosing a partial orientation of the graph and estimating the probability of a cycle. To handle dependencies between cycles, Rib\'o and Rote tried (1) selecting an independent subset of cycles, and (2) using Suen's inequality~\cite{s-ciplt-90}.
However, they could only prove an upper bound of $5{.}5202^n$ for $T(n)$, and they showed that their approach is not suitable to break the bound of $5{.}\bar{3}^n$. For the restricted problems they obtained the bounds $T_4(n)\leq 3{.}529^n$ and $T_5(n)\leq 2{.}847^n$.

Bounds for the number of spanning trees of general graphs are often expressed in terms of the vertex degree sequence of the graph. However, the main difficulty in obtaining good values for $\bound$ lies in the fact that we do not know the degree sequence of the graph in advance. Therefore, these bounds are not directly applicable. If we would assume that almost every vertex degree is $6$, which is true for the best known lower bound example presented in \cite{r-rcpps-06}, the bound of Grone and Merris~\cite{gm-bcsg-84} gives an upper bound for $T(n)$ of $(n/(n-1))^{n-1} 6^{n-1}/n$, whose asymptotic growth equals the growth rate obtained by Hadamard's inequality. To apply the more involved bound of Lyons~\cite{l-aest-03} one has to know the probabilities that a simple random walk returns to its start vertex after $k$ steps (for every start vertex). Even under the assumption that every vertex has degree $6$, it is difficult to express the  return probabilities in terms of $k$ to obtain an improvement over $6^n$.

Spanning trees are not the only interesting substructures that can be counted in planar graphs.  Aichholzer \emph{et al.}~\cite{ahhhk-npgg-07} list the known upper bounds for other subgraphs contained in a triangulation: Hamiltonian cycles, cycles, perfect matchings, connected graphs and so on. The bounds for Hamiltonian cycles and cycles have been recently improved~\cite{bkkss-ncpg-07}.

\paragraph{{\bf Overview.}}
In Section~\ref{sec:outgoing} we bound the number of spanning trees by the number of \emph{outdegree-one graphs}, i.e., the number of directed graphs obtained by picking for each vertex one outgoing edge. Cycle-free outdegree-one graphs correspond to spanning trees. Therefore we next bound the probability that a random outdegree-one graph has a cycle. For this we analyze the dependencies between cycles. In contrast to Rib\'o and Rote who showed how to avoid the dependencies in the analysis, we instead make use of the dependencies. Since our method might also find application in analyzing similar dependency structures, we phrase our probabilistic lemma in a more general setting in Section~\ref{subsec:dependencies1}. More specifically, we develop a framework to analyze a series of events for which dependent events are mutually exclusive. In Section~\ref{subsec:dependencies2} we apply this framework to bound the probability of the occurrence of a cycle. From this we derive in Sections~\ref{sec:charging} and~\ref{sec:constraints} a linear program whose objective function bounds (the logarithm of) the number of spanning trees. This linear program has infinitely many variables, and we instead consider the dual program with infinitely many constraints and present a solution in Section~\ref{sec:lp}.
%
\paragraph{\bf{Results.}}
We improve the upper bounds for the number of spanning trees of planar graphs by showing: $\bound\leq 5{.}2852$, $\bound_4\leq 3{.}4162$, and $\bound_5\leq 2{.}7156$. As a consequence the grid size needed to realize a 3d polytope with integer coordinates can now be bounded by  $O(147{.}7^n)$ instead of $O(188^n)$. For grid embeddings of simplicial 3d polytopes our results yield a small improvement to $O(27{.}94^n)$ over the old bound of $O(28{.}\bar 4^n)$.  Our improved bound for $\bound$ implies also an improved bound for the number of edge-unfolding cut-trees of a polyhedron of $n$ vertices~\cite{DDLO02}.

The maximal number of cycle-free graphs in a triangulation is another interesting quantity. Aichholzer \emph{et al.}~\cite{ahhhk-npgg-07} obtained an upper bound of $6{.}75^n$ for this number. We show in this paper 
that the improved bound for $T(n)$ yields an improved upper bound of $O(6.4948^n)$.

Multiplying $\bound$ with the number of maximal  number of triangulations a point set can have, gives an upper bound for the number of plane spanning trees on a point set. Using $30^n$ as an upper bound for the number of triangulations of a point set (obtained by Sharir and Sheffer~\cite{ss-ctpps-10}) yields an upper bound of $O(158.6^n)$ for the number of plane spanning trees on a point set. By the same construction the number of plane cycle-free graphs can be improved to $O(194{.}7^n)$. To our knowledge both bounds are the currently best known bounds.

\section{Refined outgoing edge approach}\label{sec:outgoing}
Our results are obtained by the \emph{outgoing edge approach} and its refinements. For this we consider each edge $vw$ of $G$ as a pair of directed arcs $v\to w$ and $w \to v$. Let $v_1$ be a designated vertex of $G$, and let $v_2,\ldots v_n$ be the remaining vertices.
A directed graph is called  \emph{outdegree-one}, if $v_1$ has no outgoing edge, and every remaining vertex is incident to exactly one outgoing edge. A spanning tree can be oriented as outdegree-one graphs by directing its edges towards $v_1$. This interpretation associates every spanning tree with exactly one outdegree-one graph. As a consequence the number of outdegree-one graphs contained in $G$ exceeds $t(G)$.

We can obtain all outdegree-one graphs by selecting for every vertex (except $v_1$) an edge as its outgoing edge. Let $\mathcal{S}$ be such a selection.
We denote with $d_i$ the degree of the  vertex $v_i$.  For every vertex $v_i$ we have $d_i$ choices how to select its outgoing edge. This gives us in total $\prod_{i=2}^n d_i$ different outdegree-one graphs in $G$. Due to Euler's formula the average vertex degree is less than $6$, and hence we have less than  $6^n$ outdegree-one graphs of $G$ by the geometric-arithmetic mean inequality. Thus, the outgoing edge approach gives the same bound as the strong Hadamard inequality by a very simple argument.

Outdegree-one graphs without cycles are exactly the (oriented) spanning trees of $G$. To improve the bound of $6^n$ we try to remove all graphs with cycles from our counting scheme.
Let us now consider a random selection $\mathcal{S}$ that picks the  outgoing edge for every vertex uniformly  at random. This implies that also the selected outdegree-one graph will be picked uniformly at random. Let $P_{\text{nc}}$ be the probability that the random graph selected by $\mathcal{S}$ contains no cycle. The exact number of spanning trees for any (not necessary planar) graph $G$ is given by
\begin{equation*}\label{eq:basic}
t(G)= \left( \prod_{i=2}^n d_i \right) P_{\text{nc}}.
\end{equation*}

\subsection{The dependencies of cycles in a random outdegree-one graph}\label{subsec:dependencies1}
Assume that the $t$ cycles contained in $G$ are  enumerated in some order. Notice that in an outdegree-one graph every cycle has to be directed. We consider the two orientations of a cycle with more than two vertices as one cycle.  Let $C_i$ be the  event that the  $i$-th cycle occurs and let $C_i^c$ be the event that the $i$-th cycle does not occur in a random outdegree-one graph. For events $C_i, C_j$ we denote that they are dependent by $C_i \leftrightarrow C_j$ and that they are independent by $C_i \not \leftrightarrow C_j$. We say that cycles are dependent (independent) if the corresponding events are dependent (independent).



Two cycles are independent if and only if they do not share a vertex. In turn, cycles that share a vertex are not only dependent but \emph{mutually exclusive}, i.e., they cannot occur both in an outdegree-one graph, since this would result in a vertex with two outgoing edges. This gives us the following two
properties of the events $C_i$. We say events $E_1, \ldots, E_l$ have \emph{mutually exclusive dependencies} if
$E_i \leftrightarrow E_j$ implies $\Pr [E_i \cap E_j]=0$. We say that events $E_1, \ldots, E_l$ have \emph{union-closed independencies} if $E_i \not \leftrightarrow E_{i_1}, \ldots, E_i \not \leftrightarrow E_{i_k}$ implies $E_i \not \leftrightarrow (E_{i_1} \cup \ldots \cup E_{i_k})$. It is easy to see that the events $C_i$ have mutually exclusive dependencies and union-closed independencies.
\begin{lemma}\label{lem:prob2}
 If  events $E_1, \ldots, E_l$ have mutually exclusive dependencies and union-closed independencies then for $1 < k <l$
 \[
  \Pr[\bigcap_{j=k}^{l} E_j^c \mid \bigcap_{i=1}^{k-1} E_i^c] \leq
 \prod_{j=k}^l \left( 1 - \frac{\Pr[E_j]}
 {
  \underset{\underset{E_i \leftrightarrow E_j}{1 \leq i < k:}}{\prod} \Pr [E_i^c]
 \sqrt{\underset{\underset{E_i \leftrightarrow E_j}{k \leq i \leq l:}}{\prod} \Pr [E_i^c]}
 }\right) .
 \]
\end{lemma}
Before we prove Lemma~\ref{lem:prob2} we prove the following statement.
\begin{lemma}\label{lem:prob1}
 If events $E_1, \ldots, E_l$ have mutually exclusive dependencies and union-closed independencies then
 \[
 \Pr[E_l^c \mid \bigcap_{i=1}^{l-1} E_i^c ] \leq 1 - \Pr[E_l]/ \underset{\underset{E_i \leftrightarrow E_l}{1 \leq i < l:}}{\prod} \Pr [E_i^c].
 \]
\end{lemma}
\begin{proof}
We first prove
\begin{equation}\label{eq:exactdep}
\Pr[E_l^c \mid \bigcap_{i=1}^{l-1} E_i^c ] = 1 - \Pr[E_l]/ \Pr[ \underset{\underset{E_i \leftrightarrow E_l}{1 \leq i < l:}}{\bigcap} E_i^c \mid
\underset{\underset{E_i \not \leftrightarrow E_l}{1 \leq i < l:}}{\bigcap} E_i^c].
\end{equation}
The mutually exclusive dependencies imply
\[
\Pr[E_l \mid \bigcap_{i=1}^{l-1} E_i^c ] = \Pr[E_l \cap \bigcap_{i=1}^{l-1} E_i^c ]/\Pr[\bigcap_{i=1}^{l-1} E_i^c ]=\Pr[E_l \cap \underset{E_i \not \leftrightarrow E_l}{\bigcap_{i=1,\ldots,l-1:}} E_i^c ]/\Pr[\bigcap_{i=1}^{l-1} E_i^c ].
\]
The union-closed independencies imply
\[Pr[E_l \cap \underset{E_i \not \leftrightarrow E_l}{\bigcap_{i=1,\ldots,l-1:}} E_i^c ]=\Pr[E_l] \Pr [\underset{E_i \not \leftrightarrow E_l}{\bigcap_{i=1,\ldots,l-1:}} E_i^c ],\]
which by regrouping terms implies equation~\eqref{eq:exactdep}.

A simple consequence of equation~\eqref{eq:exactdep} is that
\begin{equation}\label{eq:simpledep}
\Pr[E_l^c \mid \bigcap_{i=1}^{l-1} E_i^c ] \leq 1 - \Pr[E_l] = \Pr[E_l^c]
\end{equation}

It remains to prove that
\begin{equation*}
\Pr[ \underset{\underset{E_i \leftrightarrow E_l}{1 \leq i < l:}}{\bigcap} E_i^c \mid
\underset{\underset{E_i \not \leftrightarrow E_l}{1 \leq i < l:}}{\bigcap} E_i^c]
\leq
\Pr[
\underset{\underset{E_i \leftrightarrow E_l}{1 \leq i < l:}}{\prod} E_i^c
].
\end{equation*}
Let $i' := \max \{ 1\leq i <l \mid E_i \leftrightarrow E_l \}$. Now,
\[
\Pr[ \underset{\underset{E_i \leftrightarrow E_l}{1 \leq i < l:}}{\bigcap} E_i^c \mid
\underset{\underset{E_i \not \leftrightarrow E_l}{1 \leq i < l:}}{\bigcap} E_i^c]
=
\Pr[ E_{i'}^c \mid
\underset{\underset{i \neq i'}{1 \leq i < l:}}{\bigcap} E_i^c]
\Pr[ \underset{\underset{E_i \leftrightarrow E_l}{1 \leq i < i':}}{\bigcap} E_i^c \mid
\underset{\underset{E_i \not \leftrightarrow E_l}{1 \leq i < l:}}{\bigcap} E_i^c].
\]
Thus, by applying Inequality~\eqref{eq:simpledep} an induction yields the claim.

\end{proof}

\paragraph{Proof of Lemma~\ref{lem:prob2}.}
After rewriting $\Pr[\bigcap_{j=k}^{l} E_j^c \mid \bigcap_{i=1}^{k-1} E_i^c] = \Pr[E_l \mid \bigcap_{i=1}^{l-1} E_i^c] \Pr[\bigcap_{j=k}^{l-1} E_j^c \mid \bigcap_{i=1}^{k-1} E_i^c]$ we obtain by Lemma~\ref{lem:prob1}
\begin{equation}\label{eq:depexpanded1}
\Pr[\bigcap_{j=k}^{l} E_j^c \mid \bigcap_{i=1}^{k-1} E_i^c] =
\prod_{j=k}^l 1 - \frac{\Pr[E_j]}
 {
  \underset{\underset{E_i \leftrightarrow E_j}{1 \leq i < k:}}{\prod} \Pr [E_i^c]
  \underset{\underset{E_i \leftrightarrow E_j}{k \leq i < j:}}{\prod} \Pr [E_i^c]
 }
\end{equation}

The left hand side of equation~\eqref{eq:depexpanded1} does not depend on the order of the events $E_k, \ldots E_l$. Considering once the original order and once the reversed order of these events yields that $\Pr[\bigcap_{j=k}^{l} E_j^c \mid \bigcap_{i=1}^{k-1} E_i^c]^2$ equals
\[
\prod_{j=k}^l 1 - \frac{\Pr[E_j]}
 {
  \underset{\underset{E_i \leftrightarrow E_j}{1 \leq i < k:}}{\prod} \Pr [E_i^c]
  \underset{\underset{E_i \leftrightarrow E_j}{k \leq i < j:}}{\prod} \Pr [E_i^c]
 }
 \prod_{j=k}^l 1 - \frac{\Pr[E_j]}
 {
  \underset{\underset{E_i \leftrightarrow E_j}{1 \leq i < k:}}{\prod} \Pr [E_i^c]
  \underset{\underset{E_i \leftrightarrow E_j}{j < i \leq l:}}{\prod} \Pr [E_i^c]
 }.
\]

Thus, to prove the lemma it suffices to see that for any positive reals $a,b,c$,
$(1-a/b)(1-a/c) \leq (1-a/\sqrt{bc})^2$. By using $a/\sqrt{bc} \leq (a/b+a/c)/2$, i.e., the geometric mean of two positive numbers
is at most the arithmetic mean, we get,
\[
(1-a/b)(1-a/c) = 1 + \frac{a^2}{bc} - a/b -a/c \leq 1 + \frac{a^2}{bc} -2 a/\sqrt{bc} = (1-a/\sqrt{bc})^2.
\]
\endproof

\subsection{Bounding the probability of the appearance of cycles}\label{subsec:dependencies2}

Before estimating the probability $P_{\text{nc}}$ in terms of the vertex degrees, we introduce some notation.
A cycle of length $k$ is called a $k$-\emph{cycle}. The $k$-\emph{extension} of a cycle is the union of a cycle with all its dependent $k$-cycles. We say that the degree of a cycle is the ordered sequence of the degrees of its vertices. Let $C_{abc}$ be a $3$-cycle spanned by $v_a,v_b,v_c$, and let the degree of $C_{abc}$ be $(d_a, d_b ,d_c)=(i,j,k)$. We denote the degrees of the vertices adjacent to $v_a$ that are not part of $C_{abc}$ by the sequence $A$. In the same fashion we denote the degrees of the vertices around $v_b$ by $B$ and the around $v_c$ by $C$. The ordering in $A,B,C$ respects the counter clockwise ordering of the vertices around $v_a,v_b,v_c$ in a planar embedding. Since $G$ is planar and 3-connected  the ordering of the sequences is uniquely determined up to a global reflection~\cite{w-ig-31}. Notice that a vertex might occur in two different sequences. We call the tuple $(i,j,k,A,B,C)$, the \emph{signature} of the $2$-extension of $C_{abc}$. Similarly, we define the signature of a $2$-extension of a $2$-cycle $C_{ab}$ by the tuple $(i,j,A,B)$.
The naming convention is depicted in Figure~\ref{fig:signature}.
\begin{figure}[htbp]
\begin{center}
\begin{tabular}{cp{2cm}c}
  \includegraphics[width=.35\columnwidth]{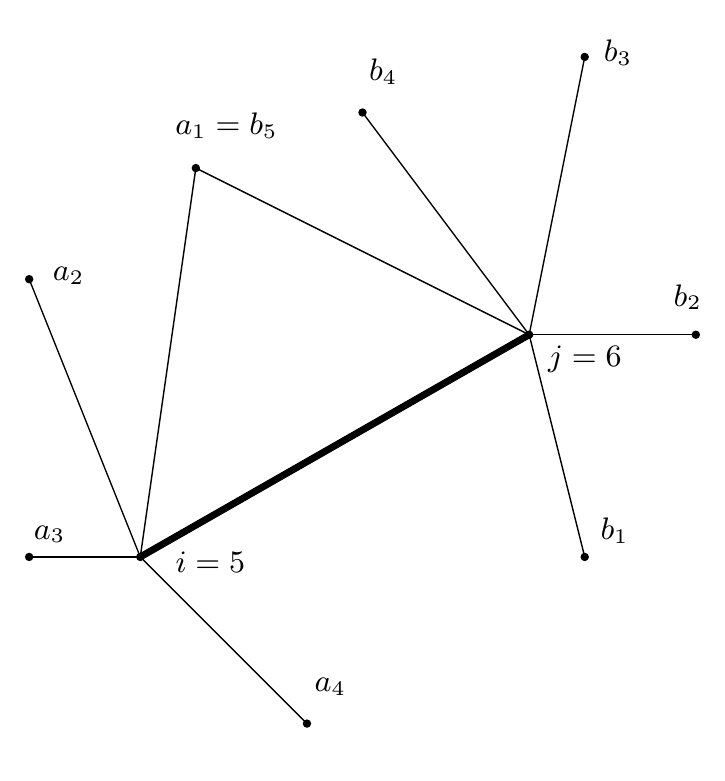} & &
  \includegraphics[width=.35\columnwidth]{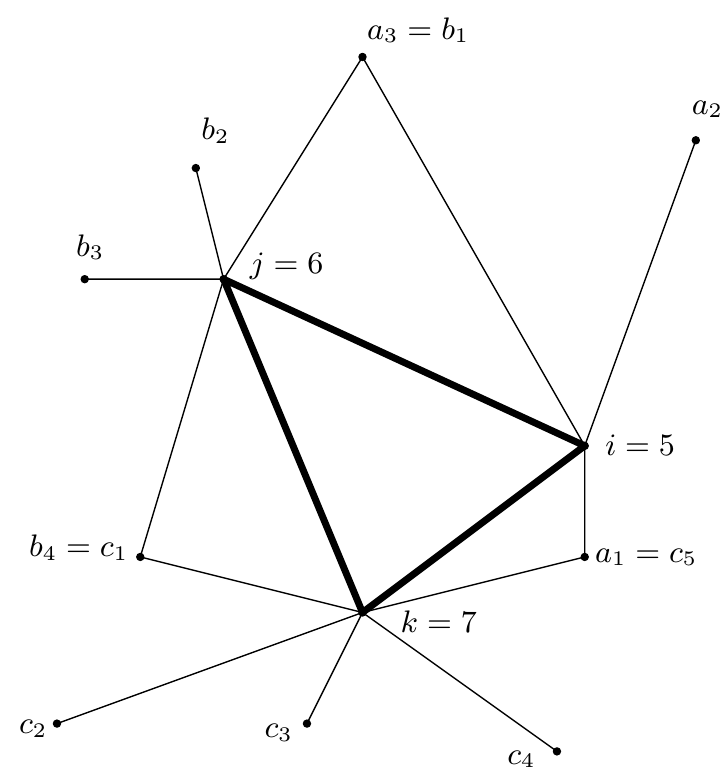} \\
  $(5,6,(a_1,\dotsc ,a_4),(b_1,\dotsc ,b_5))$ &&    $(5,6,7,(a_1,\dotsc ,a_3),(b_1,\dotsc, b_4),(c_1,\dotsc ,c_5))$
\end{tabular}
\caption{Convention for naming the signatures of $2$-extensions of $2$-cycles (on the left) and  $3$-cycles (on the right). } 
\label{fig:signature}
\end{center}
\end{figure}

We can express  $P_{\text{nc}}$ as $\Pr[\bigcap_{j=1}^t C_i^c]$. Our goal is to apply Lemma~\ref{lem:prob2} to bound this probability. As a first step we discuss how to express the number of spanning trees $t(G)$ when the distribution of signatures  of $G$ is known. Instead of $t(G)$ we bound its logarithm, i.e.,
\begin{equation}\label{eq:simple}
\log t(G)=   \sum_{i=2}^n \log d_i + \log \Pr[\bigcap_{j=1}^t C_i^c].
\end{equation}
The probability that an event $C_i$ occurs can be expressed in terms of vertex degrees. In particular,
\begin{equation*}
\begin{array}{rcll}
\Pr[C_i] &=& 1/ (d_a d_b) &\quad \text{the $i$-th cycle is a $2$-cycle on the vertices   $v_av_b$,} \\
\Pr[C_j] &=& 2/\prod_{a\colon v_a \in Z} d_a &\quad \text{the $j$-th cycle is at least a $3$-cycle on the set $Z$.}
\end{array}
\end{equation*}

The way we proceed depends on whether we are addressing the general problem (i.e., we want to bound $\bound$) or one of the restricted problems (i.e., we want to bound $\bound_4$ or $\bound_5$). In the latter case we limit our analysis to $2$-cycles only. In the general case we consider all cycles of length $2$ and cycles of length $3$ that are triangles in $G$.

We start with the general problem. Assume that all cycles $C$ are enumerated  such that the first $t_3$ cycles are the  triangles in $G$, and the last $t_2$ cycles are the $2$-cycles of $G$. In total we consider $t:=t_2+t_3$ cycles. All remaining cycles are ignored. Discarding the larger cycles gives an upper bound on $P_{\text{nc}}$ and is therefore applicable. We apply Lemma~\ref{lem:prob2} with $k=1$ and $l=t_3$ to bound $\Pr[\bigcap_{j=1}^{t_3} C_j^c]$, which is the probability that no $3$-cycle occurs. To take also the $2$-cycles into account we consider the probability that no $2$-cycle occurs under the condition that no triangle occurred as $3$-cycle, which is $\Pr[\bigcap_{j={t_3+1}}^t C_j^c | \bigcap_{j=1}^{t_3} C_j^c]$. Notice that this probability has the form stated in Lemma~\ref{lem:prob2} for $l=t$ and $k=t_3+1$. Thus, we can bound $\log \Pr[\bigcap_{j=1}^t C_j^c]$ from above by
\begin{equation}
\label{eq:loglong}
 \sum_{j=1}^{t_3}\log\Bigg( 1 - \frac{\Pr[C_j]}
 {
 \sqrt{\underset{\underset{C_i \leftrightarrow C_j}{1 \leq i \leq t_3:}}{\prod} \Pr [C_i^c]}
 }\Bigg)+\hi
 \sum_{j=t_3+1}^t \hi \log\Bigg( 1 - \frac{\Pr[C_j]}
 {
  \underset{\underset{C_i \leftrightarrow C_j}{1 \leq i < t_3+1:}}{\prod} \Pr [C_i^c]
 \sqrt{\underset{\underset{C_i \leftrightarrow C_j}{t_3 < i \leq t:}}{\prod} \Pr [C_i^c]}
 }\Bigg).
\end{equation}
Equation~\eqref{eq:loglong} is a  sum over cycles. Each summand in this sum depends only on the signature of the $2$-extension of such cycle. Hence, we can group the summands in  \eqref{eq:loglong} with identical signatures. We denote the number of $2$-extensions of $2$-cycles with signature $(i,j,A,B)$ by the variable $f_{ij}(A,B)$. Similarly, the number of $2$-extensions of $3$-cycles with signature $(i,j,k,A,B,C)$ is denoted by $f_{ijk}(A,B,C)$. In order to simplify matters, we refer to $f_{ij}(A,B)$ and $f_{ijk}(A,B,C)$ simply as $f_{ij}$ and $f_{ijk}$, or as $f$ variables.

For better readability we introduce the following notations ($X$ is used as a placeholder for $A,B,$ or $C$, and $x$ as a placeholder for $a,b,$ or $c$):
\begin{eqnarray*}
P_2(r,X) & := &  \prod\limits_{1\leq p\leq r-1} \left( 1 - \frac{1}{r x_p} \right), \quad P_3(r,X)  :=   \prod\limits_{1\leq p\leq r-2} \left( 1 - \frac{2}{r x_p x_{p+1}} \right), \\
P_{ij}(A,B)& :=&  1 - \frac{1} {ij P_3(i,A) P_3(j,B)  \Big( 1- \frac{2}{i j a_1} \Big) \Big(1 - \frac{2}{i j b_1}\Big) \sqrt{P_2(i,A) P_2(j,B)}  }, \\
P_{ijk}(A,B,C)& :=&  1\! -\! \frac{2} {ijk \sqrt{\!P_3(i,A)  P_3(j,B)  P_3(k,C)\! \Big( 1 \!- \! \frac{2}{i k a_1} \Big)\! \Big(1 \! -\! \frac{2}{i j b_1}\Big)\! \Big(1\! -\! \frac{2}{j k c_1}\Big) } }. \\
\end{eqnarray*}
We rephrase \eqref{eq:loglong} as
\begin{equation}
\begin{split}\label{eq:loglongdeg}\log \Pr[\bigcap_{j=1}^t C_j^c]
\leq&\hiii
 \sum_{i,j,k,A,B,C}\hiii f_{ijk}(A,B,C)  \log P_{ijk}(A,B,C)
 + \hii\sum_{i,j,A,B} \hi f_{ij}(A,B)  \log P_{ij}(A,B).
\end{split}
\end{equation}
The sums in the last expression  (and following similar sums) range over all feasible signatures.
Let us now consider the restricted problems. Both restricted problems are easier to analyze than the general problem, since we consider only $2$-cycles. To bound $\Pr[\bigcap_{j=1}^{t_2} C_j^c]$ we apply Lemma~\ref{lem:prob2} with $k=1$ and $l=t_2$.
Following the presentation of the general problem we define
\begin{eqnarray*}
\hat P_{ij}(A,B)& :=&  1 - \frac{1} {ij   \sqrt{P_2(i,A) P_2(j,B)}  }, \\
\end{eqnarray*}
and obtain for the restricted problems
\begin{equation}
\label{eq:loglongdegR}\log \Pr[\bigcap_{j=1}^{t_2} C_j^c]
\leq
  \sum_{i,j,A,B} f_{ij}(A,B)  \log \hat P_{ij}(A,B).
 \end{equation}

\subsection{A charging scheme for the vertex degrees}\label{sec:charging}

If we insert the bounds~\eqref{eq:loglongdeg} or \eqref{eq:loglongdegR} into equation~\eqref{eq:simple} we obtain an upper bound for $t(G)$ in terms of the signatures of $G$. However, we would like to express the first part of equation~\eqref{eq:simple}, which is  $D:=\sum_{i=1}^n \log d_i$, also in terms of the $f$ variables. For convenience we include $\log d_1$ in the sum for $D$, which is applicable since we are looking for an upper bound.

Let us first  discuss the general problem.
We split $D$ into four parts: $D_i:=\mu_i D$ for $i=1,\ldots ,4$, with $\sum_{i=1}^4 \mu_i=1$. The parameters $\mu_i$ will be fixed later.
 We express $D_1$ and $D_2$ by the $f_{ij}$ variables and   $D_3$ and $D_4$ by the $f_{ijk}$ variables.
Every vertex $v_a$ contributes   $\mu_1 \log d_a$ to  $D_1$. On the other hand, every vertex $v_a$ is part of $d_a$ $2$-cycles. We charge the total amount of $\mu_1 \log d_a$ uniformly  to these $2$-cycles. Thus, every $2$-cycle incident to  $v_a$ gets $\mu_1 \log d_a / d_a$ from $v_a$.
In a similar fashion we charge $D_2$ to the $2$-extension of  $2$-cycles. Let  $v_a v_b$ be an edge in $G$ and let $v_r\not=v_b$ be a  vertex adjacent to $v_a$. Distributing $\mu_2 \log  d_r$  uniformly, assigns every $2$-extension with ``endpoint'' $v_r$ the fraction of $\mu_2 \log  d_r / ( d_r(d_a -1))$ from $ v_r$. For $D_3$ and $D_4$ we argue analogously. We can therefore express $D$ by
\begin{equation}\label{eq:D}
\begin{split}
D_1=&  \mu_1 \sum_{i,j,A,B} f_{ij}(A,B)  \Big(\frac{\log i}{i}+\frac{\log j}{j}\Big), \\
D_2 =& \mu_2 \sum_{i,j,A,B} f_{ij}(A,B) \Big( \sum_{a_r \in A} \frac{\log a_r}{a_r(i-1)}+\sum_{b_r \in B} \frac{\log b_r}{b_r(j-1)}\Big), \\
D_3=&  \mu_3 \sum_{i,j,k,A,B,C} f_{ijk}(A,B,C)  \Big(\frac{\log i}{i}\!+\!\frac{\log j}{j}\!+\!\frac{\log k}{k}\Big), \\
D_4 =& \mu_4 \hiii \sum_{i,j,k,A,B,C} \hiii f_{ijk}(A,B,C) \Big( \sum_{a_r \in A} \frac{\log a_r}{a_r(i-1)}\!+\hi\sum_{b_r \in B} \frac{\log b_r}{b_r(j-1)}\!+\hi\sum_{c_r \in C} \frac{\log c_r}{c_r(k-1)}\Big).
\end{split}
\end{equation}
We can now express $\log P_{\text{nc}}$  as sum over all signatures. This sum can be subdivided into one part that contains the $f_{ij}$ variables and one part that contains the $f_{ijk}$ variables.
The part that considers the $2$-cycles is given by
\begin{align*}
 D1+D2  +\sum_{i,j,A,B} f_{ij}(A,B)  \log P_{ij}(A,B),  \tag{G2}\label{eq:G2} \\
\end{align*}
and the part that considers the $3$-cycles is given by
\begin{align*}
 D3+D4  +  \sum_{i,j,k,A,B,C} f_{ijk}(A,B,C)  \log P_{ijk}(A,B,C).  \tag{G3}\label{eq:G3} \\
\end{align*}
For the restricted problems we only have $2$-cycles. Using bound~\eqref{eq:loglongdegR} and setting $\mu_3=\mu_4=0$, we can bound the number of spanning trees by
\begin{align*}
 D1+D2  +\sum_{i,j,A,B} f_{ij}(A,B)  \log\hat P_{ij}(A,B)  \tag{R2}\label{eq:R2}.
\end{align*}

\subsection{Finding constraints}\label{sec:constraints}
In this section we construct \emph{necessary} conditions for the $f$  variables that have to hold for planar graphs with  $n$ vertices. We reuse the ideas from  the charging scheme in Section~\ref{sec:charging}. Instead of giving every vertex $\log d_i$ to distribute, we assign to every vertex  an amount of~$1$. This gives us a total of $n$ units. Following the construction of the equations of \eqref{eq:D} we obtain
\begin{flalign*}
 \sum_{i,j,A,B} f_{ij}(A,B)  \Big(\frac{1}{i}+\frac{1}{j}\Big)&=n, \tag{A2} \label{A2} \\
\sum_{i,j,k,A,B,C} f_{ijk}(A,B,C)  \Big(\frac{1}{i}+\frac{1}{j}+\frac{1}{k}\Big)& =n,   \tag{A3} \label{eq:A3}  \\
 \sum_{i,j,A,B} f_{ij}(A,B) \Big( \sum_{a_r \in A} \frac{1}{a_r(i-1)}+\sum_{b_r \in B} \frac{1}{b_r(j-1)}\Big)&=n,   \tag{B2} \label{eq:B2} \\
\hii \sum_{i,j,k,A,B,C}\hiii f_{ijk}(A,B,C) \Big(  \! \sum_{a_r \in A} \frac{1}{a_r(i-1)} \!+ \!\sum_{b_r \in B} \frac{1}{b_r(j-1)}\!+\!\sum_{c_r \in C} \frac{1}{c_r(k-1)}\Big)& =n   \tag{B3} \label{eq:B3} .
\end{flalign*}
Another set of constraints is given by the number of $2$-cycles and $3$-cycles a planar graph can have, which is related to the number of edges and faces of $G$. Every $2$-cycle is counted by some $f_{ij}$ variable, hence the sum over all $f_{ij}$ equals the number of edges, which we name $m$. Since we consider only $3$-cycles of triangles, the sum of the $f_{ijk}$ variables equals the number of triangles, which for a planar graph is at most $2n$.  We obtain
\begin{align*}
 \sum_{i,j,A,B} f_{ij}(A,B) & \leq m, \tag{C2} \label{C2} \\
\sum_{i,j,k,A,B,C} f_{ijk}(A,B,C)& \leq 2n.   \tag{C3} \label{eq:C3}  \\
\end{align*}
For the general case we have $m \leq 3n$, for the restricted case where quadrilaterals are allowed we have $m \leq 2n$, and for the remaining case we have $m\leq 5n/3 $. All these bounds can be obtained by a simple double counting argument using Euler's formula. As trivial condition we restrict the $f$ variables to be non-negative.

The constraints so far might be fulfilled by a signature that does not come from a planar graph. In particular, the degree sequence of the graph induced by the cycles might be unrelated to the degree sequence of the graph induced by the $2$-extensions. To overcome this ambiguity we consider the number of edges with vertex degree $i$ at one vertex and degree $j$ at the other. Let this number be $n_{ij}$. Clearly, we have that $n_{ij}=\sum_{A,B} f_{ij}(A,B)$, where the sum ranges about all feasible sequences $A,B$. On the other hand, $n_{ij}$ can be counted by its appearances in the $2$-extensions of $2$-cycles. Every edge with degree $(i,j)$ will show up in $(i-1)+(j-1)$  $2$-extensions. Let $\chi_i(X)$ denote the number of appearances of $i$ in the sequence $X$. We can express $((i-1)+(j-1))n_{ij}$ as
$\sum_{k,A,B} f_{ik}(A,B) \chi_j(A) + \sum_{k,A,B} f_{kj}(A,B) \chi_i(B). $
This leads us to a new constraint of the form
\begin{align*}
(i+j-2) \sum_{A,B} f_{ij}(A,B) & = \sum_{k,A,B} f_{ik}(A,B) \chi_j(A) + \sum_{k,A,B} f_{kj}(A,B) \chi_i(B). \label{eq:Eij} \tag{E$ij$}
\end{align*}
In  the case where the smallest face of the graph is a pentagon, we were able to improve the solution of the linear program  by adding the constraint (E33). Other constraints of the form \eqref{eq:Eij} gave no improvement.

\section{Solving the linear programs}\label{sec:lp}
For the general case and a fixed $n$ we can solve the  problem of maximizing $t(G)$ while fulfilling the constraints of Section~\ref{sec:constraints} by solving two linear programs. The first problem, which is derived from the $2$-cycles, is given by the objective function \eqref{eq:G2}, and the constraints (A2-C2). The problem associated with the $3$-cycles consists of the objective function \eqref{eq:G3}, and the constraints (A3-C3). Our goal, however, is not to solve these problems for a fixed $n$, but to study the solution in terms of $n$. To get an expression in terms of  $n$, we normalize  the $f$ variables. Instead of letting $f$ count absolute numbers, we consider $f$ as ratio between these absolute numbers and $n$. This allows us to cancel $n$ in every constraint and objective function.

By removing $n$ we lost a natural bound on the largest vertex degree, which is $n-1$. Moreover, we cannot bound the number of signatures and have to deal with infinitely many variables. To overcome this difficulty we  look at the dual program. The dual has up to four variables, which we name $\lambda_1$,$\lambda_2$,$\lambda_3$, and $\lambda_4$.
Since we cannot bound the number of variables in the primal, we have infinitely many dual constraints. By the weak linear programming duality it suffices to find a point in the feasible area of the dual, because every such point is an upper bound for the primal solution. Appendix~A lists all relevant dual programs. Let $Z_2$ be the LP solution for (G2) and $Z_3$ be the LP solution for (G3). The combined bound for $\bound$ can be computed by adding $\exp(Z_2)$ and $\exp(Z_3).$

We use a two step approach to find the solution of the dual programs. First we solve the dual program with a finite number of the constraints. We expect that the maximum number of spanning trees will be realized on a graph with evenly distributed vertex degrees. Therefore, we expect in the general case that the variables $f_{ij}$, with $i,j$ far away from $6$ will be zero, and hence the corresponding dual constraints will not hold with equality. For the restricted problems we expect values for $i,j$ around $3$ and $4$.
We use these assumptions to construct (finite) linear programs that will most likely contain the constraints that determine the dual solution.
We prove later that the other constraints are also fulfilled.
The candidates for the dual solution that were obtained by the finite programs are listed in Table~\ref{tbl:lpresults}. The table also lists our choice of  $\mu$ values. We picked these particular numbers because they have been experimentally proven to be useful in the  later analysis.
\renewcommand\arraystretch{1.1}%
\begin{table}
\centering
\begin{tabular}{|l|c|c|c|}
\hline
LP instance & $\mu$ values &  \multirow{1}{*}{solution} &  \multirow{1}{*}{$\bound$} \\
\hline
general problem ($2$-cycles) & \multirow{2}{*}{$\mu_1=0{.}3,\mu_2=0{.}25$}&$\lambda_2=0{.}180948$ &   \multirow{5}{*}{$5{.}28515$} \\
 (G2)$\to\max$, s.t. (A2-C2), $m=3$& &  $\lambda_3=0{.}232445$ & \\
 \cline{1-3}
 general problem ($3$-cycles)  &\multirow{3}{*}{$\mu_3=0{.}225,\mu_4=0{.}225$}&$\lambda_1=0{.}0980332$   &   \\
  (G3)$\to\max$, s.t. (A3-C3) & &  $\lambda_2=0{.}192612$ &\\
  & & $\lambda_3=0{.}247984$ & \\
 \hline
 restricted problem (no triangles) & \multirow{2}{*}{$\mu_1=0{.}9,\mu_2=0{.}1$}& \multirow{2}{*}{$\lambda_3=0{.}614264$}&   \multirow{2}{*}{$3{.}41619$}  \\
 (R2)$\to\max$, s.t.  (A2-C2), $m=2$ &  & & \\
\hline
 restricted problem (no  3,4-gons) & \multirow{3}{*}{$\mu_1=0{.}615,\mu_2=0{.}385$}& $\lambda_1=-0{.}615054$ &  \multirow{3}{*}{$2{.}71567$}  \\
 (R2)$\to\max$, s.t.  (A2-C2,E33)&  & $\lambda_3=0{.}744706$ &  \\
 $m= 5/3$ && $\lambda_4=0{.}001954$  & \\
  \hline
\end{tabular}
\label{tbl:lpresults}
\caption{The results of the dual linear programs. Not specified $\lambda$ values are zero.}
\end{table}

The verifications of the LP solutions is tedious. We redirect the interested reader for a complete analysis to the \emph{Mathematica} scripts, downloadable under  \url{http://wwwmath.uni-muenster.de/u/schuland/scripts-nst.zip}. For the 2-cycle part of the general problem we present the general idea behind the analysis in this place. Remarks for the analysis of the other programs can be found in Appendix~B.

\begin{lemma}\label{lem:verification}
For $\mu_1=0{.}3$, $\mu_2=0{.}25$ and $\lambda_1=0$, $\lambda_2=0{.}180948$, $\lambda_3=0{.}232445$, and $m=3$ all dual constraints
\emph{(A2-C2)} for feasible signatures are fulfilled.
\end{lemma}

\paragraph{Proof outline.} The complete proof including technical details and algorithms for brute force testing of some cases can be found in the \textit{Mathematica} scripts.

We have to show that  the left-hand side of  \eqref{eq:dual} is negative.
Let us ignore the $\log$-term for now -- it is negative, and therefore it suffices to show that \eqref{eq:dual} without the $\log$-term is negative.
Let $a$ be some vertex degree stored in the sequences $A$ or $B$. We notice that the ``effect'' of $a$ in  \eqref{eq:dual} is $\mu_2 \log(a) /(a(i-1)) -  \lambda_2 /(a(i-1))$. For integers this term is maximized for $a=6$, and for $a\leq 5$ monotonically increasing. Hence, it is sufficient to show  that \eqref{eq:dual} holds for all $a\geq 6$.

Using the lower bound $a=6$ gives us

\begin{multline} \mu_2\left(\sum_{a \in A} \frac{\log a}{a(i-1)}+\sum_{b \in B} \frac{\log b}{b(j-1)}\right) -\lambda_2\left(\sum_{a \in A} \frac{1}{a(i-1)}+\sum_{b \in B} \frac{1}{b(j-1)}\right) \\ \leq 2 \mu_2  (\log 6)/6 - 2 \lambda_2 /6.\end{multline}
We know  that $\lambda_1=0$, and that $(\log j)/j$   is maximized for integers at $j=3$. Hence, if the expression $ \mu _1((\log i)/i + (\log 3 )/3) + 2\mu _2(\log 6)/6 -2\lambda_2 /6- \lambda_3$ is negative, the dual constraint is fulfilled. By considering the partial derivative in $i$, we see that this holds for all $i\geq 31$ (and a symmetric argument implies the same for $j\geq 31$). For all other pairs $ij$ we check the dual constraint by hand using again the lower bound $6$ for the entries of $A$ and $B$. By this we can eliminate more cases and are left with $83$ tuples (see Appendix~B).

We will now consider the $\log$-term of  \eqref{eq:dual} again. Therefore, we have to obtain new upper and lower bounds for the entries of $A$ and $B$. Let $a$ be a vertex degree that is stored in $A$ or $B$.
It can be observed that the $\log$-term in terms of  $a$ is increasing. Since the function without $\log$-term was increasing between $3$ and $6$ the function with $\log$-term is increasing as well in this range and hence the maximum has to be attained on a value larger than $6$.

To obtain an upper bound on the entries of $A$ and $B$ we argue as follows: We first assume that all entries of $A$ and $B$ equal $6$, except for some $a_x \in A$. We look at  the derivative of the constraint  \eqref{eq:dual} in $a_x$ and observe that the maximum is realized at some value $a_x'$. We study now what happens if we change one of the values, say $a_y$, which we assumed to be $6$. Let $g(a_x,a_y)$ be the dual constraint in $a_y$ and $a_x$. It can be shown that  $ \partial^2 g(a_x,a_y)/\partial a_x\partial a_y$ is negative. Hence, for $a_y>6$,

\[0<\frac{\partial g(a_x,6)}{\partial a_x}-\frac{\partial g(a_x,a_y)}{\partial a_x}=\frac{(\partial g(a_x,6)- g(a_x,a_y))}{\partial a_x}.\]

In other words, the differences between $g(a_x,6)$ and $g(a_x,a_y)$ are increasing in $a_x$. Let $a'_y$ be fixed and  $d= g(a'_x,6)-g(a'_x,a'_y)$. Clearly the function $h(a_x,a'_y):=g(a_x,a'_y)+d$ attains its extrema on the same positions as $g(a_x,a'_y)$ and its differences to $g(a_x,6)$ are also increasing in $a_x$. Since $h(a'_x,a'_y)=g(a'_x,6)$, the graph of $h(a_x,a'_y)$ is always below $g(a_x,6)$ for $a_x>a'_x$. Therefore, no maximum can be attained on $h(a_x,a'_y)$ for $a_x>a'_x$ and hence the smallest upper bound is achieved for  $a_y=6$. Hence $a'_x$ is a valid upper bound for $a_x$. The computed bounds for every remaining pair $ij$ can be found in  Appendix~B.

As last step we check all possible signatures that are left to test. Notice that there are still many cases open. We observe that the ordering of the degrees in $A$ and $B$  does not matter, except for  $P_3(i,A) P_3(j,B)  \Big( 1- \frac{2}{i j a_1} \Big) \Big(1 - \frac{2}{i j b_1}\Big)$. But this expression can be bounded by using the maximum in $A$ for all $a_x$, and the maximum in $B$ for all $b_x$. With this estimation we can try all sequences with \emph{sorted} sequences $A$ and $B$. This eliminates most of the cases. There are three pairs left to test, namely $(6,6)$, $(5,6)$, and $(5,7)$. We test the constraints for the remaining signatures by checking \eqref{eq:dual} for all combinations.

The solutions for the linear programs lead to the main theorem.
\begin{theorem}\label{thm:main}
Let $G$ be a planar graph with $n$ vertices. The number of spanning trees of $G$ is at most $O(5{.}28515^n)$. If $G$ is 3-connected and contains no triangle, then the number of its spanning trees is bounded by $O(3{.}41619^n)$. If $G$ is 3-connected and contains no triangle and no quadrilateral, then the number of its spanning trees is bounded by $O(2{.}71567^n)$.
\end{theorem}

\section{Further bounds and future work}
The results of Theorem~\ref{thm:main} improve several related upper bounds.
Using the observations by Rib\'o \emph{et al.}~\cite{rrs-epsg-07} we obtain the following bounds for grid embeddings of 3d polytopes.
\begin{corollary}
Let $G$ be the graph of a $3$-polytope $\mathcal{P}$ with $n$ vertices. $\mathcal{P}$ admits a realization as combinatorial equivalent polytope with integer coordinates and
\begin{enumerate}
\item no coordinate larger than $O(147{.}7^n)$,
\item no coordinate larger than $O(39{.}9^n)$, if $G$ contains a quadrilateral,
\item no coordinate larger than $O(28{.}4^n)$, if $G$ contains a triangle.
\end{enumerate}
\end{corollary}


The number $F(n)$ of cycle-free graphs in a planar graph with $n$ vertices is bounded by the number of selections of at most $n-1$ edges from the graphs~\cite{ahhhk-npgg-07}. Thus, $F(n)\leq \sum_{k=0}^{n-1} \binom{3n-6}{k}$.
For $0 \leq q \leq 1/2$ we have that $\sum_{i=0}^{\lfloor q m\rfloor} \binom{m}{qm}<2^{H(q)m}$,    where $H(q):=-\log(q)^q-\log(1-q)^{(1-q)}$ is the binary entropy (see for example~\cite[page 427]{gf-pct-06}). This shows that, $F(n)<6{.}75^n$ by setting $m=3n$ and $q=1/3$.

We give a better bound based on the bound for the number of spanning trees. We first bound the number $F(n,k)$ of forests in $\mathcal{G}_n$ with $k$ edges. On one hand,
the above argument yields an upper bound of $F (n,k) \leq  f_1(k) := \binom{3n-6}{k}$. On the other hand, every forest with $k$ edges can be constructed by selecting $k$ edges from a spanning tree of $\mathcal{G}_n$. This gives as alternative bound $F(n,k) \leq f_2(k) =: \binom{n-1}{k}T(n)$.
Now, the number of cycle-free graphs is bounded by
\[
F(n) =\sum_{k=0}^{n-1} F(n,k) \leq n \max_{0\leq k <n} F(n,k) \leq n \max_{0\leq q<1} \min (f_1(q n),f_2(q n)).
\]
We use $\binom{n}{qn} \leq 2^{n H(q)}$ as upper bound for the binomial coefficients (see for example \cite[page 1097]{clrs-ia-01}) to obtain
\[
f_1(q n) < \hat f_1(q n) :=2^{3n H(q/3)} \quad \mbox{ and } \quad
f_2(q n) < \hat f_2(q n) := T(n) 2^{nH(q)}.
\]
The computed maximal value for the minimum of $\hat f_1$ and $\hat f_2$ is realized at $qn= 0{.}94741\; n$.
This yields a bound of $n\cdot 6{.}4948^n$ for the number of cycle-free graphs. For the computation of these values we used numerical methods. Observe that  $\hat f_1(qn)$ realizes $6{.}4948^n$ at  a larger value $q$ than $\hat f_2(qn)$. The correctness of the numerical computations follows from the monotonicity of $\hat f_1$ and $\hat f_2$ in $(n/2,n]$.
%
%
For the number of plane spanning trees and cycle-free graphs on a planar point set, we obtain improved upper bounds by multiplying our bounds with the bound of $O(30^n)$ on the maximum number of triangulations on a planar point set~\cite{ss-ctpps-10}.
\begin{theorem}
The number of cycle-free graphs in a planar graph with $n$ vertices is bounded by $n\cdot 6{.}4948^n$. The number of plane
spanning trees on $n$ planar points is in $O(158.6^n)$, the number of cycle-free graphs in $O(194{.}7^n)$.
\end{theorem}

%

We expect better bounds for the number of cycle-free graphs in a planar graph from a more direct application of the outgoing edge approach. By adding a new vertex that is linked to a subset of the other vertices, every cycle-free graph can be turned into a spanning tree of the augmented graph.  Without excluding any cycles we get a bound of  $7^n$.
Under the assumption that almost every vertex has degree $6$, the refined outgoing edge method would yield a bound of $6{.}5027^n$ when excluding 2-cycles and of $6{.}4244$ when excluding 2 and 3-cycles.
So far we were not able to  check all constraints of the corresponding linear programs.

We finish our presentation with a discussion on how one could improve our results further. Since we consider  only $2$-cycles and $3$-cycles from triangles, one would obtain a better bound for $P_{\text{nc}}$ by taking also larger cycles into account. We do not expect to win anything by considering $3$-cycles that are not triangles, because in the lower bound example (the wrapped up triangular grid)  all $3$-cycles are triangles.
The analysis using larger cycles is more complicated and needs an extensive case distinction. Furthermore, we expect that there would be too many cases left for the brute force check. From our perspective, the following refinement  seems tractable: Beside the $2$-cycles, and $3$-cycles on triangles, we also analyze $4$-cycles that belong to two triangles sharing an edge (the diagonal). The $4$-cycles can be analyzed by extending the events $C_i$ for the $2$-cycles to the following event: the $i$-th $2$-cycle occurs, or the corresponding $4$-cycle, whose diagonal is associated with the $i$-th cycle occurs. Assuming that the solution of the corresponding linear program is given by having almost every vertex degree $6$, this would lead to $\bound=5.25603$. Since the resulting linear program is more complicated, the verification of the dual solutions is tedious.

Notice that Lemma~\ref{lem:prob2} uses two enumerations of the events $C_i$ to avoid the influence of the ordering. An elaborated enumeration scheme of the events $C_i$ might give better bounds. Furthermore, we could consider ``extension of extensions'' to analyze larger locally connected pieces of the graph at once. This results in a powerful but very complicated incarnation of the outgoing edge approach.

The reader might think, that additional constraints in the linear programs might improve the outcome of our analysis. However, we expect that the solutions of the dual programs give the correct distribution of signatures. In particular, the solutions the dual programs match the candidates for the lower bound examples that were presented in~\cite{r-rcpps-06}.

\paragraph{Acknowledgements:} We thank G\"unter Rote for suggesting this problem to us and for many inspiring and fruitful discussions on this subject.

\bibliographystyle{abbrv}
\bibliography{stpg}

\begin{thebibliography}{10}

\bibitem{ahhhk-npgg-07}
O.~Aichholzer, T.~Hackl, C.~Huemer, F.~Hurtado, H.~Krasser, and B.~Vogtenhuber.
\newblock On the number of plane geometric graphs.
\newblock {\em Graph. Comb.}, 23(1):67--84, 2007.

\bibitem{bhn-lif-97}
R.~Bacher, P.~de~la Harpe, and T.~Nagnibeda.
\newblock The lattice of integral flows and the lattice of integral cuts on a
  finite graph.
\newblock {\em Bull. Soc. Math. de France}, 125:167--198, 1997.

\bibitem{b-aptg-97}
N.~Biggs.
\newblock Algebraic potential theory on graphs.
\newblock {\em Bull. London Math. Soc.}, 29:641--682, 1997.

\bibitem{b-cfcg-99}
N.~Biggs.
\newblock Chip-firing and the critical group of a graph.
\newblock {\em J. Algebraic Combin.}, 9:25--46, 1999.

\bibitem{bkkss-ncpg-07}
K.~Buchin, C.~Knauer, K.~Kriegel, A.~Schulz, and R.~Seidel.
\newblock On the number of cycles in planar graphs.
\newblock In G.~Lin, editor, {\em COCOON}, volume 4598 of {\em Lecture Notes in
  Computer Science}, pages 97--107. Springer, 2007.

\bibitem{clrs-ia-01}
T.~H. Cormen, C.~E. Leiserson, R.~L. Rivest, and C.~Stein.
\newblock {\em Introduction to Algorithms}.
\newblock {MIT} Press, second edition, 2001.

\bibitem{DDLO02}
E.~D. Demaine, M.~L. Demaine, A.~Lubiw, and J.~O'Rourke.
\newblock Enumerating foldings and unfoldings between polygons and polytopes.
\newblock {\em Graphs and Combinatorics}, 18(1):93--104, 2002.

\bibitem{gf-pct-06}
J.~Flum and M.~Grohe.
\newblock {\em Parameterized Complexity Theory}.
\newblock Springer, 2006.

\bibitem{gm-bcsg-84}
R.~Grone and R.~Merris.
\newblock A bound for the complexity of a simple graph.
\newblock {\em Discrete Mathematics}, 69(1):97--99, 1988.

\bibitem{hj-ma-85}
R.~A. Horn and C.~R. Johnson.
\newblock {\em Matrix Analysis}.
\newblock Cambridge University Press, 1985.

\bibitem{l-aest-03}
R.~Lyons.
\newblock Asymptotic enumeration of spanning trees.
\newblock {\em Combinatorics, Probability {\&} Computing}, 14(4):491--522,
  2005.

\bibitem{m-strg-83}
B.~D. McKay.
\newblock Spanning trees in regular graphs.
\newblock {\em Euro. J. Combinatorics}, 4:149--160, 1983.

\bibitem{r-rcpps-06}
A.~{Rib\'o Mor}.
\newblock {\em Realization and Counting Problems for Planar Structures: Trees
  and Linkages, Polytopes and Polyominoes}.
\newblock PhD thesis, Freie Universit\"at Berlin, 2006.

\bibitem{rrs-epsg-07}
A.~{Rib\'{o} Mor}, G.~Rote, and A.~Schulz.
\newblock Embedding 3-polytopes on a small grid.
\newblock In J.~Erickson, editor, {\em Symposium on Computational Geometry},
  pages 112--118. ACM, 2007.

\bibitem{rg-rsp-96}
J.~Richter-Gebert.
\newblock {\em Realization Spaces of Polytopes}, volume 1643 of {\em Lecture
  Notes in Mathematics}.
\newblock Springer, 1996.

\bibitem{r-nstpg-05}
G.~Rote.
\newblock The number of spanning trees in a planar graph.
\newblock In {\em Oberwolfach Reports}, 2. European Mathematical Society,
  Publishing House, 2005.

\bibitem{ss-ctpps-10}
M.~Sharir and A.~Sheffer.
\newblock Counting triangulations of planar point sets.
\newblock {\em Manuscript}, 2010.

\bibitem{s-emw-22}
E.~Steinitz.
\newblock Encyclop\"adie der mathematischen {W}issenschaften.
\newblock In {\em Polyeder und Raumteilungen}, pages 1--139. 1922.

\bibitem{s-ciplt-90}
S.~Suen.
\newblock A correlation inequality and a poisson limit theorem for
  nonoverlapping balanced subgraphs of a random graph.
\newblock {\em Random Struct. Algorithms}, 1(2):231--242, 1990.

\bibitem{t-crg-60}
W.~T. Tutte.
\newblock Convex representations of graphs.
\newblock {\em Proceedings London Mathematical Society}, 10(38):304--320, 1960.

\bibitem{t-hdg-63}
W.~T. Tutte.
\newblock How to draw a graph.
\newblock {\em Proceedings London Mathematical Society}, 13(52):743--768, 1963.

\bibitem{w-ig-31}
H.~Whitney.
\newblock A set of topological invariants for graphs.
\newblock {\em Amer. J. Math.}, 55:235--321, 1933.

\end{thebibliography}
\vspace*{\fill}
\pagebreak[4]

\section*{Appendix A: Dual programs}
\subsubsection*{General Problem ($2$-cycle part):}
\[\text{Minimize } \lambda_1+\lambda_2+ 3\lambda_3,\]
such that, $\lambda_3\geq 0$, and for all signatures $(i,j,A,B)$:
  \begin{multline}\label{eq:dual}
\log P_{ij}(A,B) + \mu_1  \Big(\frac{\log i}{i}+\frac{\log j}{j}\Big) + \mu_2  \Big( \sum_{a_r \in A} \frac{\log a_r}{a_r(i-1)} +\sum_{b_r \in B} \frac{\log b_r}{b_r(j-1)}\Big)\\ -  \lambda_1 \Big(\frac{1}{i}+\frac{1}{j}\Big) - \lambda_2   \Big( \sum_{a_r \in A} \frac{1}{a_r(i-1)}+\sum_{b_r \in B} \frac{1}{b_r(j-1)}\Big) -\lambda_3 \leq 0.
 \end{multline}
 \subsubsection*{General Problem ($3$-cycle part):}
\[\text{Minimize } \lambda_1+\lambda_2+ 2\lambda_3,\]
such that, $\lambda_3\geq 0$, and for all signatures $(i,j,k,A,B,C)$:
 \begin{multline*}\hspace*{-2ex}
 \log P_{ijk}(A,B,C) + \mu_3\!  \Big(\frac{\log i}{i}+\frac{\log j}{j}+\frac{\log k}{k}\Big) +\mu_4\!\Big( \!\sum_{a_r \in A}\! \frac{\log a_r}{a_r(i-1)}+\!\sum_{b_r \in B}\! \frac{\log b_r}{b_r(j-1)}+\!\sum_{c_r \in C}\! \frac{\log c_r}{c_r(k-1)}\Big) \\
 -\lambda_1 \Big(\frac{1}{i}+\frac{1}{j}+\frac{1}{k}\Big) -\lambda_2 \Big( \sum_{a_r \in A} \frac{1}{a_r(i-1)}+\sum_{b_r \in B} \frac{1}{b_r(j-1)}+\sum_{c_r \in C} \frac{1}{c_r(k-1)}\Big) -\lambda_3 \leq 0
 \end{multline*}
\subsubsection*{Restricted Problem (no triangles):}
\[\text{Minimize } \lambda_1+\lambda_2+ 2 \lambda_3,\]
such that, $\lambda_3\geq 0$, and for all signatures $(i,j,A,B)$:
 \begin{multline*}\hspace*{-2ex}
\log\hat P_{ij}(A,B) + \mu_1  \Big(\frac{\log i}{i}+\frac{\log j}{j}\Big) + \mu_2  \Big( \sum_{a_r \in A} \frac{\log a_r}{a_r(i-1)} +\sum_{b_r \in B} \frac{\log b_r}{b_r(j-1)}\Big)\\ -  \lambda_1 \Big(\frac{1}{i}+\frac{1}{j}\Big) - \lambda_2   \Big( \sum_{a_r \in A} \frac{1}{a_r(i-1)}+\sum_{b_r \in B} \frac{1}{b_r(j-1)}\Big) -\lambda_3 \leq 0.
 \end{multline*}
 \subsubsection*{Restricted Problem (no triangles, no quadrilaterals):}
\[\text{Minimize } \lambda_1+\lambda_2+ 5/3 \: \lambda_3,\]
such that, $\lambda_3\geq 0$, and for all signatures $(i,j,A,B)$:
 \begin{multline*}\hspace*{-2ex}
\log\hat P_{ij}(A,B) + \mu_1  \Big(\frac{\log i}{i}+\frac{\log j}{j}\Big) + \mu_2  \Big( \sum_{a_r \in A} \frac{\log a_r}{a_r(i-1)} +\sum_{b_r \in B} \frac{\log b_r}{b_r(j-1)}\Big)\\ -  \lambda_1 \Big(\frac{1}{i}+\frac{1}{j}\Big) - \lambda_2   \Big( \sum_{a_r \in A} \frac{1}{a_r(i-1)}+\sum_{b_r \in B} \frac{1}{b_r(j-1)}\Big) -\lambda_3 - \lambda_4 E_{ij}(A,B) \leq 0,
 \end{multline*}
 \[\text{where } E_{ij}(A,B)=\begin{cases} 0& \text{if $i\geq 4$},\\
                                 -\chi_3(A) & \text{if $i=3, j\geq 4$},\\
                                  4-  \chi_3(A) - \chi_3(B) & \text{if $i=3, j=3$}. \end{cases}
 \]
\vspace*{\fill}
\pagebreak[4]

\section*{Appendix B: Verification of the solutions of the dual linear programs}
We give in this place informations how the candidates for the dual solutions in Table~\ref{tbl:lpresults} can be checked. The complete analysis was computed by a series of \textsl{Mathematica} scripts, which can be found \url{http://wwwmath.uni-muenster.de/u/schuland/scripts-nst.zip}. These scripts also perform the necessary brute force part of our analysis. We follow the general, more detailed, procedure of verifying the dual solutions  presented in the proof of Lemma~\ref{lem:verification}. We discuss the four linear programs one by one.

\paragraph*{General case ($2$-cycle part):} This case is  covered in detail in the proof of Lemma~\ref{lem:verification}. We only list in
Table~\ref{tbl:bf-general2}
the intermediate results of the analysis, i.e., the cases which need to be checked  brute force. The table contains also the upper bound for the entries of $A$ and $B$ for each case. Notice that the calculations  for the bounds of the first entry $a_1$ of $A$ and of the first entry $b_1$ of $B$ are slightly different.
\begin{table}[htdp]
\begin{center}
\begin{tabular}{cp{1mm}cp{1mm}c}
\begin{tabular}{ c   c@{\hspace{1ex}} c@{\hspace{1ex}} c@{\hspace{1ex}}c  }
\hline
 & & \multicolumn{3}{c}{ upper bounds for} \\
$i$ & $j$ & $a_1/b_1$ &  $a\in A$ &  $b \in A$ \\
\hline
 3 & 3 & 10 & 9 & 9 \\
 3 & 4 & 9 & 8 & 8 \\
 4 & 4 & 8 & 8 & 8 \\
 3 & 5 & 8 & 7 & 8 \\
 4 & 5 & 7 & 7 & 8 \\
 5 & 5 & 7 & 7 & 7 \\
 3 & 6 & 8 & 7 & 8 \\
 4 & 6 & 7 & 7 & 7 \\
 5 & 6 & 7 & 7 & 7 \\
 6 & 6 & 7 & 7 & 7 \\
 3 & 7 & 7 & 7 & 8 \\
 4 & 7 & 7 & 7 & 7 \\
 5 & 7 & 7 & 7 & 7 \\
 6 & 7 & 7 & 7 & 7 \\
 7 & 7 & 7 & 7 & 7 \\
 3 & 8 & 7 & 7 & 8 \\
 4 & 8 & 7 & 7 & 7 \\
 5 & 8 & 7 & 7 & 7 \\
 6 & 8 & 7 & 7 & 7 \\
 7 & 8 & 7 & 7 & 7 \\
  8 & 8 & 7 & 7 & 7 \\
 3 & 9 & 7 & 7 & 8 \\
 4 & 9 & 7 & 7 & 7 \\
 5 & 9 & 7 & 7 & 7 \\
 6 & 9 & 7 & 7 & 7 \\
 7 & 9 & 7 & 7 & 7 \\
 8 & 9 & 7 & 7 & 7 \\
 9 & 9 & 7 & 7 & 7 \\
  \hline
\end{tabular} & &

\begin{tabular}{ c   c@{\hspace{1ex}} c@{\hspace{1ex}} c@{\hspace{1ex}}c  }
\hline
 & & \multicolumn{3}{c}{ upper bounds for} \\
$i$ & $j$ & $a_1/b_1$ &  $a\in A$ &  $b\in B$ \\
\hline

 3 & 10 & 7 & 7 & 8 \\
 4 & 10 & 7 & 7 & 7 \\
 5 & 10 & 7 & 7 & 7 \\
 6 & 10 & 7 & 7 & 7 \\
 7 & 10 & 7 & 7 & 7 \\
 8 & 10 & 7 & 7 & 7 \\
 3 & 11 & 7 & 7 & 8 \\
 4 & 11 & 7 & 7 & 7 \\
 5 & 11 & 7 & 7 & 7 \\
 6 & 11 & 6 & 7 & 7 \\
 7 & 11 & 6 & 7 & 7 \\
 3 & 12 & 7 & 7 & 8 \\
 4 & 12 & 7 & 7 & 7 \\
 5 & 12 & 7 & 7 & 7 \\
 6 & 12 & 6 & 7 & 7 \\
 7 & 12 & 7 & 7 & 7 \\
 3 & 13 & 7 & 7 & 8 \\
 4 & 13 & 7 & 7 & 7 \\
 5 & 13 & 6 & 7 & 7 \\
 6 & 13 & 7 & 7 & 7 \\
 3 & 14 & 7 & 7 & 8 \\
 4 & 14 & 7 & 6 & 7 \\
 5 & 14 & 6 & 6 & 7 \\
 6 & 14 & 7 & 6 & 7 \\
 3 & 15 & 7 & 6 & 8 \\
 4 & 15 & 6 & 6 & 7 \\
 5 & 15 & 7 & 6 & 7 \\
 6 & 15 & 6 & 6 & 7 \\
 \hline
\end{tabular} &&

\begin{tabular}{ c   c@{\hspace{1ex}} c@{\hspace{1ex}} c@{\hspace{1ex}}c  }
\hline
 & & \multicolumn{3}{c}{ upper bounds for} \\
$i$ & $j$ & $a_1/b_1$ &  $a \in A$ &  $b\in A$ \\
\hline
 3 & 16 & 7 & 7 & 8 \\
 4 & 16 & 6 & 6 & 7 \\
 5 & 16 & 7 & 7 & 7 \\
 3 & 17 & 6 & 6 & 8 \\
 4 & 17 & 7 & 7 & 7 \\
 5 & 17 & 6 & 7 & 7 \\
 3 & 18 & 6 & 7 & 8 \\
  4 & 18 & 7 & 7 & 7 \\
 5 & 18 & 6 & 6 & 7 \\
 3 & 19 & 7 & 7 & 8 \\
 4 & 19 & 6 & 6 & 7 \\
 3 & 20 & 7 & 6 & 8 \\
 4 & 20 & 6 & 6 & 7 \\
 3 & 21 & 6 & 6 & 7 \\
 4 & 21 & 6 & 6 & 7 \\
 3 & 22 & 7 & 6 & 7 \\
 4 & 22 & 6 & 6 & 7 \\
 3 & 23 & 6 & 6 & 7 \\
 4 & 23 & 6 & 6 & 7 \\
 3 & 24 & 6 & 6 & 7 \\
 4 & 24 & 6 & 6 & 7 \\
 3 & 25 & 6 & 6 & 7 \\
 3 & 26 & 6 & 6 & 7 \\
 3 & 27 & 6 & 6 & 7 \\
 3 & 28 & 6 & 6 & 7 \\
 3 & 29 & 6 & 6 & 7 \\
 3 & 30 & 6 & 6 & 7\\
 \hline
 &&&&
\end{tabular}

\end{tabular}
\caption{Cases for the brute force test in the $2$-cycle part of the general problem.}
\label{tbl:bf-general2}
\end{center}
\end{table}%

\paragraph*{General case ($3$-cycle part):} The lower bound for the entries of the entries of $A,B,C$ can be obtained by  maximizing $(\mu_2 \log a  - \lambda_2)/a$, which is maximized for integers at $a=6$. To get an upper bound for the $i,j,k$ values compute first the maximum of $ \mu _1\log x/x-\lambda _1/x=:\delta$. Knowing $\delta$ and the lower bound for for the entries of $A,B,C$, we are able to compute the the upper bound for the $i,j,k$ values as the maximum of $\mu _1(\log i /i + 2\log \delta/ \delta ) + 3\mu _2\log 6/6 -\lambda _1(1/i+2/\delta)- 3 \lambda_2 /6 - \lambda_3$. This expression is negative for all integers greater than $13$, hence we have to test all remaining tuples for $i,j,k\leq13$.
These cases are listed  in~\eqref{eq:bf-general3} as set of triplets $i,j,k$. For all $i,j,k$ triplets we obtained that $7$ is an upper bound for the entries in $A,B,C$.
{\small
\begin{multline}\label{eq:bf-general3}
\{(4,4,4),(4,4,5),(4,5,5),(5,5,5),(4,4,6),(4,5,6),(5,5,6),(4,6,6),(5,6,6),(6,6,6),\\
(4,4,7),(4,5,7),(5,5,7),(4,6,7),(5,6,7),(6,6,7),(4,7,7),(5,7,7),(6,7,7),(7,7,7), \\
(4,4,8), (4,5,8),(5,5,8),(4,6,8),(5,6,8),(6,6,8),(4,7,8),(5,7,8),(6,7,8),(7,7,8), \\
(4,8,8),(5,8,8), (6,8,8),(4,4,9),(4,5,9),(5,5,9),(4,6,9),(5,6,9),(6,6,9),(4,7,9), \\
(5,7,9),(6,7,9),(4,8,9),(5,8,9),(4,4,10),(4,5,10),(5,5,10),(4,6,10),(5,6,10), \\
(6,6,10),(4,7,10),(5,7,10),(4,4,11),(4,5,11),(5,5,11),(4,6,11),(5,6,11),(4,4,12), \\
(4,5,12),(5,5,12),(4,6,12), (4,4,13),(4,5,13),(3,4,5),(3,5,5),(3,4,6),(3,5,6), \\
(3,6,6), (3,4,7),(3,5,7),(3,6,7),(3,7,7), (3,4,8),(3,5,8),(3,6,8),(3,7,8),(3,4,9), \\
(3,5,9),(3,6,9), (3,4,10),(3,5,10),(3,4,11)\}
\end{multline}
}

\paragraph*{Restricted cases:}
For the restricted cases we use the natural lower bound of $3$ for the entries of $A$ and $B$. If the smallest face of $G$ is a quadrilateral, an upper bound for $i,j$ can be computed by maximizing $\mu _1((\log i)/i + \log  (3)/3) + 2\mu _2(\log 3)/3 - \lambda_3$. This expression is maximized for integers when $i=10$, and hence negative for $i\geq 10$. If the smallest face of $G$ is a pentagon, we have to compute the maximum of  $\mu _1 (\log x)/x-\lambda_1 /x=:\delta$ first. Using $\delta$ for $j$ we can show that $\mu _1(\log i)/i + (\log \delta)/\delta) + 2\mu _2(\log 3)/3 -\lambda_1(1/i+1/\delta)- 2\lambda_2 /3 - \lambda_3$ is negative for all $i\geq 11$ . Thus were are left with trying out all cases for $i,j < 10$ (if $G$ contains no triangle) and $i,j < 11$ (if $G$ contains no triangle and quadrilateral). The upper bounds of the entries of $A$ and $B$ are computed for each of these cases as explained in the proof of Lemma~\ref{lem:verification}.
\begin{table}[htdp]
\begin{center}
\begin{tabular}{cp{2cm}c}
\begin{tabular}{@{\hspace*{2ex}} c@{\hspace*{2ex}}  @{\hspace*{2ex}} c@{\hspace*{2ex}} c }
\hline

$i$ & $j$ &  upper bounds\\
& &  for $a\in A \cup B$  \\
\hline
 3 & 3 & 5 \\
 3 & 4 & 5 \\
 4 & 4 & 4 \\
 3 & 5 & 4 \\
 4 & 5 & 4 \\
 5 & 5 & 4 \\
 3 & 6 & 4 \\
 4 & 6 & 4 \\
 5 & 6 & 4 \\
 3 & 7 & 4 \\
 4 & 7 & 4 \\
 3 & 8 & 4 \\
 4 & 8 & 4 \\
 3 & 9 & 4 \\
  \hline
\end{tabular} &&
\begin{tabular}{@{\hspace*{2ex}} c@{\hspace*{2ex}}  @{\hspace*{2ex}} c@{\hspace*{2ex}}c  }
\hline
$i$ & $j$ &  upper bounds \\
& & for $a \in A\cup B$  \\
\hline
 3 & 3 & 4 \\
 3 & 4 & 4 \\
 4 & 4 & 3 \\
 3 & 5 & 4 \\
 4 & 5 & 3 \\
 5 & 5 & 3 \\
 3 & 6 & 4 \\
 4 & 6 & 3 \\
 5 &  6 & 3 \\
 3 & 7 & 4 \\
 4 & 7 & 3 \\
 3 & 8 & 4\\
 4 & 8 & 3 \\
 3 & 9 & 4 \\
 3 & 10 & 4\\
  \hline
\end{tabular} \\
no triangles & & no triangles and 4-gons
\end{tabular}
\caption{Cases for the brute force test of the restricted problems.}
\label{tbl:bf-restricted}
\end{center}
\end{table}%

The ``pentagon case'' contains one more subtlety. The constraints where $i$ or $j$ are $3$ are slightly more complicated because we included (E33) in the primal program. For convenience  we assume that the upper bound on the entries of $A$ and $B$ for these $ij$ values is at least $4$. This simplifies the computations since for a distinct neighbor $a_i$ the influence of (E33) is a constant when  $a_i$ varies between $4$ and $\infty$.
See
Table~\ref{tbl:bf-restricted}
for the remaining cases. Notice that we can assume that the sequences $A$ and $B$ are ordered since the order of $A$ and $B$ is not relevant for the dual constraints of the restricted cases.
\vspace*{\fill}
\pagebreak[4]

\end{document}